\pdfoutput=1
\documentclass[12pt,a4paper,oneside]{amsart}
\usepackage{color,amssymb,latexsym,amsfonts,textcomp,fancyhdr,calc,graphicx}
\usepackage{amsmath,amstext,amsthm,amssymb,amsxtra}
\usepackage[left=2cm,right=2cm,bottom=2cm,top=2cm]{geometry} %%%%%%%
\usepackage[colorlinks,citecolor=red,pagebackref,hypertexnames=false]{hyperref}
\usepackage{dsfont}
\usepackage[framemethod=tikz]{mdframed}
\usepackage[inline]{asymptote}
\usepackage{wrapfig}
\usepackage{tcolorbox}

\usepackage[colorlinks,citecolor=red,pagebackref,hypertexnames=false]{hyperref}

\usepackage{txfonts,pxfonts,tikz} %also txfonts 
\usepackage{tgschola}
\usepackage[T1]{fontenc}

\theoremstyle{plain}
\newtheorem{theorem}{Theorem}%[section]
\newtheorem{question}{Question}
\theoremstyle{definition}

\theoremstyle{remark}
\newtheorem{remark}{Remark}

\newenvironment{que}
    {\begin{tcolorbox}[colback=cyan]\begin{question}
    }
    { 
    \end{question}\end{tcolorbox}
    }

\newcommand{\beql}[1]{\begin{equation}\label{#1}}
\newcommand{\eeq}{\end{equation}}

\newcommand{\Abs}[1]{{\left|{#1}\right|}}

\newcommand{\Ceil}[1]{{\left\lceil{#1}\right\rceil}}

\newcommand{\Set}[1]{{\left\{{#1}\right\}}}

\newcommand{\RR}{{\mathbb R}}

\newcommand{\ZZ}{{\mathbb Z}}

\newcommand{\QQ}{{\mathbb Q}}

\newcommand{\one}{{\mathds{1}}}

\newcommand{\supp}{{\rm supp\,}}

\newcommand{\vol}{{\rm vol\,}}
\newcommand{\ft}[1]{\widehat{#1}}

\newcommand{\diam}{{\rm diam\,}}

\newcommand{\qq}[1]{{{\begin{que} {#1} \end{que}}}}

\newcounter{rem}
\makeatletter
\@namedef{subjclassname@2020}{\textup{2020} Mathematics Subject Classification}
\makeatother

\newcounter{mysec}
\def\themysec{\arabic{mysec}}
\newcommand{\mysection}[1]{
  \vskip 0.25in
  \refstepcounter{mysec}
  \addcontentsline{toc}{section}{\themysec.\ {#1}}
  \noindent{\bf \S\themysec.\ {#1}}
}

\begin{document}

\title{Simultaneous tiling}

\author{Mihail N. Kolountzakis}
\address{Department of Mathematics and Applied Mathematics, University of Crete, Voutes Campus, 70013 Heraklion, Crete, Greece.}
\email{kolount@gmail.com}

\subjclass[2020]{primary 52C22;  secondary 20K99}
\keywords{Lattices, tiling, common fundamental domain, Steinhaus problem}

\date{\today}

\thanks{Supported by the Hellenic Foundation for Research and Innovation, Project HFRI-FM17-1733 and by grant No 4725 of the University of Crete.}

\begin{abstract}
We discuss problems of simultaneous tiling. This means that we have an object (set, function) which tiles space with two or more different sets of translations. The most famous problem of this type is the Steinhaus problem which asks for a set simultaneously tiling the plane with all rotates of the integer lattice as translation sets.
\end{abstract}

\maketitle

\centerline{\em Dedicated to the memory of Dimitris Gatzouras}

\tableofcontents

\setlength{\parskip}{0.5em}

\sloppy

\mysection{Tiling by translation.}
For the purposes of this paper \footnote{This is a survey paper which contains no new results. Many of the results and questions in this paper are from \cite{kolountzakis2021functions} on parts of which this paper is heavily based.} tiling is by translation only \cite{kolountzakis2004milano}. We have an object $T$ (the \textit{tile}) which may be a set or a function on some abelian group $G$ (usually the Euclidean space but it may be $\ZZ^d$ or a finite group)
\begin{figure}[h]
%\begin{center}
\centerline{\includegraphics[width=0.2\linewidth]{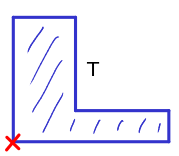}}
\caption{An $L$-shaped tile. The red point is the origin}
%\end{center}
\end{figure}
which we are translating around by a set of translations $\Lambda$, in such a way that everything in the group $G$ is covered exactly once, with the possible exception of a set of zero Haar measure, to account for such irrelevant things such as boundaries overlapping, which we generally do not care about.
\begin{figure}[h]
{\includegraphics[width=0.4\textwidth]{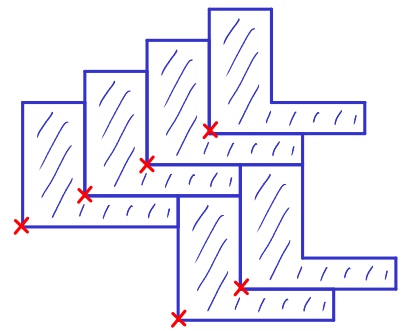}}
\caption{A translational tiling by the $L$-shaped region. The red points are the translation set.}
\end{figure}
One convenient way to define tiling by a function $f$ (which can be an indicator function, if we want tiling by a set) when translated at the locations $\Lambda$ is to demand that
\beql{tiling}
\sum_{\lambda\in\Lambda} f(x-\lambda) = \text{const.},
\eeq
for almost all $x \in G$. To avoid most issues of convergence it makes sense to ask that $f \ge 0$, though some interesting problems do arise with signed $f$ \cite{kolountzakis2021tiling}.

\mysection{Tiling in Fourier space.} It is easy to see that \eqref{tiling} may be rewritten as a convolution
\beql{tiling-conv}
f*\delta_{\Lambda} = \text{const.}
\eeq
where $\delta_\Lambda = \sum_{\lambda\in\Lambda} \delta_\lambda$ is the measure that encodes the locations $\Lambda$ by placing a unit mass on each of them.
\begin{figure}[h]
\includegraphics[width=0.2\linewidth]{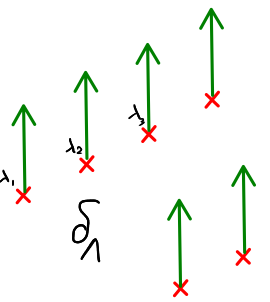}
\caption{The collection $\delta_\Lambda$ of point masses that encodes the set $\Lambda$}
\end{figure}
Taking the Fourier Transform of this we obtain
\beql{tiling-ft}
\ft{f} \ft{\delta_\Lambda} = C \delta_0.
\eeq
This implies that the tempered distribution $\ft{\delta_\Lambda}$ is supported on the zeros of $\ft{f}$ plus the origin
\beql{tiling-ft-support}
\supp\ft{\delta_\Lambda} \subseteq \Set{\ft{f} = 0} \cup \Set{0}.
\eeq
Let us now restrict ourselves to the case of $G = \RR^d$ and $\Lambda \subseteq \RR^d$ being a lattice $\Lambda = A \ZZ^d$, where $A$ is a non-singular $d\times d$ matrix. The Poisson summation formula reads
$$
\ft{\delta_\Lambda} = \frac{1}{\Abs{\det A}}\delta_{\Lambda^*}
$$
in this case, where $\Lambda^* = A^{-\top}\ZZ^d$ is the dual lattice of $\Lambda$, so the tiling of $f$ with $\Lambda$ becomes equivalent to
$$
\ft{f}(\lambda^*) = 0\ \ \text{ for all } \lambda^* \in \Lambda^*\setminus\Set{0}.
$$
\begin{figure}[h]
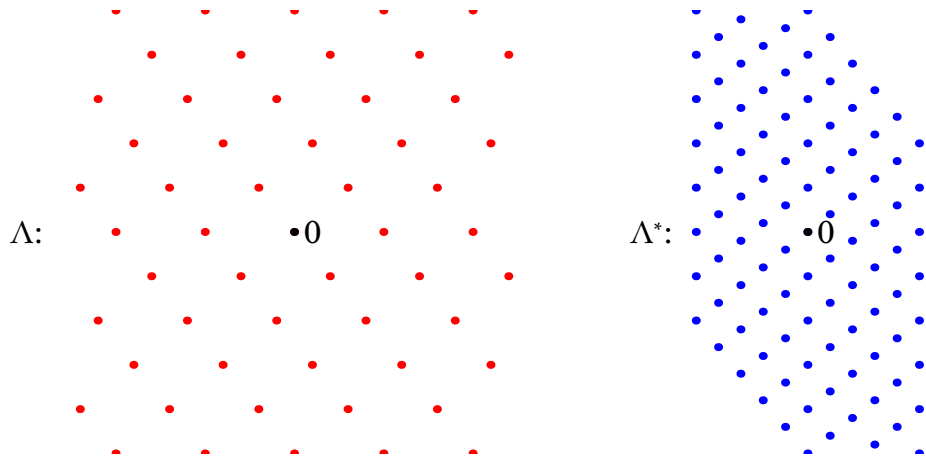

\begin{asy}
size(12cm);

picture pic, picd;
real[][] a={{2, 1.2}, {0, 1}};
real[][] b=transpose(inverse(a));

int N=5;

for(int i=-N; i<=N; ++i)
 for(int j=-N; j<=N; ++j) {
  real[] x={i, j};
  real[] y=a*x;
  pair p=(y[0], y[1]);
  dot(pic, p, red);
  y = b*x;
  pair p=(y[0], y[1]);
  dot(picd, p, blue);
 }

dot(pic, "0", (0, 0), black);
dot(picd, "0", (0, 0), black);

clip(pic, (-N, -N)--(N, -N)--(N, N)--(-N, N)--cycle);
label(pic, "$\Lambda$:", (-1.2*N, 0));
clip(picd, (-N, -N)--(N, -N)--(N, N)--(-N, N)--cycle);
label(picd, "$\Lambda^*$:", (-0.7*N, 0));
add(pic);
add(shift((2.3*N, 0))*picd);
\end{asy}

\caption{A lattice $\Lambda$ and its dual $\Lambda^*$}
\end{figure}
\mysection{The Steinhaus tiling problem.} In the \textit{Steinhaus tiling problem} we are seeking a tile that can tile {\em simultaneously} with many different sets of translations. The most important case is: can we find a subset of the plane which can tile (by translations) with all rotates of the integer lattice $\ZZ^2$? In some sense we are asking for a set in the plane that can behave simultaneously like all these rotated squares (Fig.\ \ref{rotates}).
\begin{figure}[h]
\includegraphics[width=0.3\linewidth]{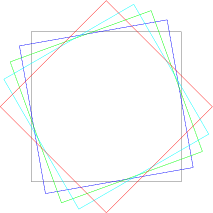}
\caption{The rotated squares are fundamental domains of all rotates of $\ZZ^2$}
\label{rotates}
\end{figure}
There are two major variations of the Steinhaus problem: the \textit{measurable} and the \textit{set-theoretic} case. In the measurable case we demand our tile to be a Lebesgue measurable subset of $\RR^d$ and we are, at the same time, relaxing our requirements and are allowing a subset of measure 0 of space not to be covered exactly once by the translates of the tile. In the set-theoretic case we allow the tile to by any subset and we typically ask that every point is covered exactly once, allowing no exceptions.

Komj\'ath \cite{komjath1992lattice} answered the Steinhaus question in the affirmative in $\RR^2$ when tiling by all rotates of the set $B=\ZZ\times\Set{0}$
showing that there are such Steinhaus sets (but such a set $A$ cannot be measurable as was shown
recently in \cite{kolountzakis2017measurable}).
Sierpinski \cite{sierpinski1958probleme} showed that a bounded set $A$ which is either closed or open cannot have the
lattice Steinhaus property (that is, intersect all rigid motions of $\ZZ^2$ at exactly one point -- another way to say that $A$ tiles precisely with all rotates of $\ZZ^2$).
Croft \cite{croft1982three} and Beck \cite{beck1989lattice} showed that no bounded and measurable set $A$ can
have the lattice Steinhaus property (but see also \cite{mallinikova1995}).
Kolountzakis \cite{kolountzakis1996problem,kolountzakis1996new} and Kolountzakis and Wolff \cite{kolountzakis1999steinhaus}
proved that any measurable set in the plane that has the measurable Steinhaus property must necessarily
have very slow decay at infinity (any such set must have measure 1). In \cite{kolountzakis1999steinhaus}
it was also shown that there can be no measurable Steinhaus sets in dimension $d\ge 3$ (tiling with all rotates $\rho\ZZ^d$, where $\rho$ is in the full orthogonal group)
a fact that was also shown later by Kolountzakis and Papadimitrakis \cite{kolountzakis2002steinhaus} by a very different method.
See also \cite{chan2007steinhaus,mauldin2003comments,ciucu1996remark,srivastava2005steinhaus}. Kolountzakis \cite{kolountzakis1997multi} looks at the case where
we are only asking for our set to tile with {\em finitely many} lattices, not all rotates as in the original problem, which we are also doing in this paper.
In a major result Jackson and Mauldin \cite{jackson2002sets,jackson2002lattice} proved the existence of Steinhaus sets in the plane which tile with all rotates of $\ZZ^2$
(not necessarily measurable). Their method does not extend to higher dimension $d \ge 3$. See also \cite{mauldin2001some,jackson2003survey}.
It was also shown in \cite{kolountzakis2017measurable} that a set $A$ which tiles with all rotates of a finite set $B$ cannot be measurable.

\mysection{The Steinhaus problem in Fourier space.} Most of the results on the measurable Steinhaus problem start by observing that if $E \subseteq \RR^2$ is Steinhaus then every rotate $R_\theta E$ of $E$ tiles with $\ZZ^2$, which means that for every angle $\theta$ the Fourier Transform $\ft{\one_{R_\theta E}}$ vanishes on $\ZZ^2 \setminus\Set{0}$ since $\ZZ^2$ is the dual lattice of itself. This implies that $\ft{\one_E}$ vanishes on all rotates of $\ZZ^2$. In other words $\ft{\one_E}$ vanishes on all circles centered at the origin that go through at least one integer lattice point.
\begin{figure}[h]
\resizebox{7cm}{!}{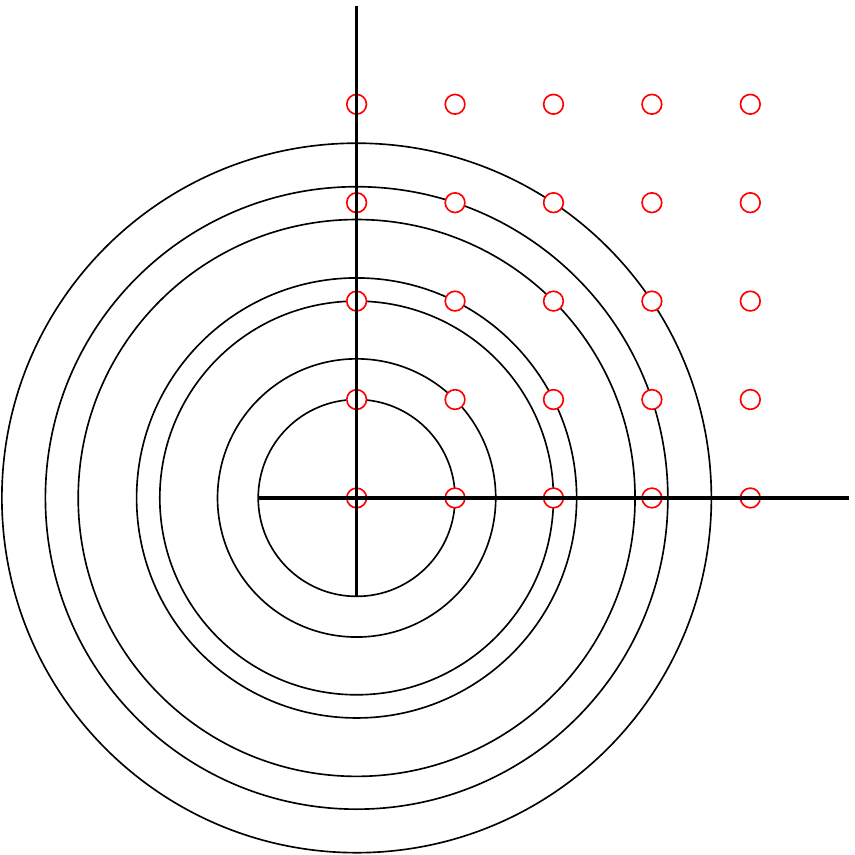}
\caption{The Fourier Transform of any Steinhaus set must vanish on all circles centered at the origin that go through at least one integer lattice point}
\end{figure}
The number of these circles is large. There are a little less than $O(R^2)$ such circles of radius $\le R$. Many zeros of a function sometimes imply decay at infinity, and, by the usual uncertaintly principle (both $f$ and $\ft{f}$ cannot decay fast at infinity), since $\ft{\one_E}$ is small at infinity it follows that $\one_E$ is large (e.g., $E$ cannot be bounded).

\mysection{Allowing functions instead of sets in the Steinhaus problem.} Let us now relax our requirements and allow our tile to be a function instead of a set (instead of indicator function, in other words).
Satisfying the requirements of the Steinhaus tiling problem with a function is generally much easier than with a set. The problem becomes interesting only if one asks for further properties that this function should have. Therefore we try to find a function with {\em small support}, or to prove that the support of such a function must necessarily be large. Asking for $f$ to have a small support goes against $f$ having the ability to tile space, especially with many different sets of translations $T$. The reason is that for $f$ to tile by translations with $T$ its Fourier transform must contain a rich set of zeros \cite{kolountzakis2004milano}. This set of zeros must be able to support the Fourier transform of the measure $\delta_T = \sum_{t \in T} \delta_t$ (which encodes the set of translations). By the well known uncertainty principle in harmonic analysis a rich set of zeros for $\ft{f}$ usually requires (in various different senses) a large support for $f$ \cite{havin1994uncertainty}.

It is very easy to take $\ft{f}$ to vanish on the required circles, but one must do it in a way that ensures that $f$ is itself {\em small} in some sense, such as the diameter of its support or the volume of its support.

\mysection{Small diameter of the support: lower bounds.}
The first thing that comes to mind is to take $f$ to be a convolution. It takes a moment to verify that if $f$ tiles with a set of translates $T$ then so does $g*f$ for any $g \in L^1(\RR^d)$. One can either verify this by checking the definition of tiling for $g*f$ or observe that tiling is a condition that can be checked on the Fourier side \cite{kolountzakis2004milano} and $\ft{g*f} = \ft{g} \cdot \ft{f}$ has an even richer set of zeros that $\ft{f}$.

So, since $\ft{f}$ has to vanish on the dual lattices $\Lambda_i^*\setminus\Set{0}$ we can take
\beql{convolution}
f = \one_{D_1} * \one_{D_2} * \cdots * \one_{D_N},
\eeq
where $D_i$ is a fundamental parallelepiped of $\Lambda_i$. Since $D_i + \Lambda_i$ is a tiling it follows that $\widehat{\one_{D_i}}$ vanishes on $\Lambda_i^*\setminus\Set{0}$ and that $f$ vanishes on their union and hence tiles with all $\Lambda_i$. This can be slightly generalized by taking, instead if the indicator functions $\one_{D_i}$ any function $f_i$ that tiles with $\Lambda_i$
\beql{convolution-of-functions}
f = f_1 * f_2 * \cdots f_N.
\eeq

\begin{figure}[h]
\includegraphics[width=0.4\linewidth]{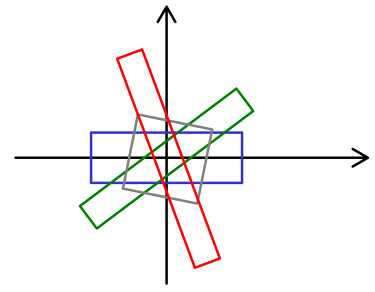}
\caption{The fundamental doamins of several lattices. A constant fraction of them project to a set of large diameter onto one of the coordinate axes.}
\label{fig:rotated-fds}
\end{figure}
The following observation (see detailed proof in \cite{kolountzakis2021functions}) was already made in \cite{kolountzakis1999steinhaus} in the case $f_i = \one_{D_i}$.
\begin{theorem}\label{th:convolution-lb}
If $\Lambda_1, \ldots, \Lambda_N$ are lattices in $\RR^d$ of volume $c_1 \le \vol\Lambda_i$ and $f = f_1 * f_2 * \cdots * f_N$ then
\beql{convolution-lower-bound}
\diam\supp f \ge C_d N.
\eeq
\end{theorem}

The reason is that a constant fraction of the supports of the $f_i$ project onto a constant fraction of their diameter onto \textit{some} line, say one of the axes. This implies (obvious if the $f_i$ are nonnegative; one needs the Titchmarsh convolution theorem in the general case) that so does the support of the convolution $f = f_1*\cdots*f_N$ (shown in Fig.\ \ref{fig:rotated-fds} for the $f_i$ being the indicator functions of fundamental parallelepipeds of the lattices).

If the lattices $\Lambda_i$ satisfy some ``roundness'' assumption, e.g. if each $\Lambda_i$ is assumed to have a fundamental domain of diameter bounded independent of $N$ (as in the important case when all the lattices are rotates of $\ZZ^d$), then the convolution tile \eqref{convolution} has diameter which is also at most $C\cdot N$.

On the other hand we have the following rather general lower bound for the diameter of the support \cite{kolountzakis1999steinhaus} assuming only a certain genericity assumption \eqref{no-intersection} on the $\Lambda_i$.
\begin{theorem}\label{kw-lb}
If $\Lambda_1, \ldots, \Lambda_N \subseteq \RR^d$, $d \ge 1$, are lattices of volume equal to 1 such that
\beql{no-intersection}
\Lambda_i \cap \Lambda_j = \Set{0}\ \  \text{ for all $i \neq j$,}
\eeq
then if $f$ tiles with all these lattices we have
\beql{lb}
\diam\supp f \ge C_d N^{1/d}.
\eeq
\end{theorem}

%\begin{remark}
%In \cite{kolountzakis1999steinhaus} Lemma \ref{kw-lb} was proved under the assumption that all lattices have volume {\em equal} to 1. The proof works verbatim for volume merely bounded below by a constant independent of $N$.
%\end{remark}

The main question is therefore:

\qq{
Can the gap between the lower bound \eqref{lb} and the linear upper bound $O(N)$ achievable by the convolution tile \eqref{convolution} (in the case of ``round'' lattices, having fundamental domains bounded in diameter by a constant) be bridged?

Are there examples of lattices $\Lambda_i$, $i=1,2,\ldots,N$, satisfying \eqref{no-intersection} and a non-zero function $f \in L^1(\RR^d)$ that tiles with all $\Lambda_i$ and such that
$$
\diam\supp f = o(N)?
$$

In other words, do there exist collections of lattices for which a common tile $f$ can be found which is diameter-wise more efficient than the convolution construction \eqref{convolution}?
}

\mysection{A case of large diameter.} We observe now \cite{kolountzakis2021functions} that for some collections of lattices the linear upper bound cannot be improved.
The lattices given are both ``round'' (have a fundamental domain bounded independent of $N$) and satisfy the genericity assumption \eqref{no-intersection}. There are however \textit{collinearities} so, in some sense, this is not a generic situtation. 
\begin{theorem}\label{th:long-common-tiles}
For $d \ge 1$ and for each $N$ there are lattices $\Lambda_1, \ldots, \Lambda_N \subseteq \RR^d$, of volume 1, such that if $f \in L^1(\RR^d)$, $\int f \neq 0$, tiles with all of them then
$$
\diam\supp f \ge C_d N.
$$
\end{theorem}
\begin{proof}
We give the proof in the case $d=2$. It works with obvious changes in all dimensions $d > 2$ and it is even easier in dimension $d=1$.
\begin{figure}[h]
\includegraphics[width=0.4\linewidth]{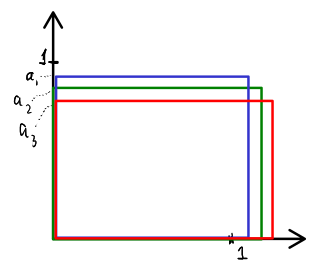}
\caption{The fundamental rectangles of the lattices of Theorem \ref{th:long-common-tiles} which have only very long common tiles.}
\label{fig:aligned}
\end{figure}

Take $\Lambda_i^*$ to be generated by the two vectors
$$
u_i = (0, a_i), v_i = (1/a_i, 0),
$$
where the numbers $a_1, \ldots, a_N$ are linearly independent over $\QQ$ and
$$
0.9 < a_i < 1.
$$
If $f$ tiles with all $\Lambda_i$ then $\ft{f}$ vanishes on all points of the form
$$
(0, k \cdot a_i),\ \ i=1, 2, \ldots, N,\ \ k \in \ZZ\setminus\Set{0}.
$$
Since all these points are different it follows that the density of zeros on the $y$-axis is $\ge C \cdot N$. This implies that
$$
\diam\supp \pi_2(f) \ge C \cdot N
$$
(say, by Jensen's formula) where $\pi_2(f)$ is the one-variable function
$$
\pi_2(f)(y) = \int_{\RR} f(x, y) \,dx.
$$
(This is not an identically zero function by our assumption on the integral of $f$.)
This in turn implies
$$
\diam\supp f \ge C \cdot N.
$$
\end{proof}

\mysection{Small volume of the support.}
Another measure of the size of the support is its volume. Can we construct a common tile $f$ for the lattices $\Lambda_i$ such that $\Abs{\supp f}$ is small?

In the case of $f$ given by \eqref{convolution} it is clear that
$$
\supp{f} = D_1 + D_2 + \ldots + D_N.
$$
To keep things concrete let us assume that all $\Abs{D_i} = 1$ in \eqref{convolution} (unimodular lattices). Then the Brunn-Minkowski inequality \cite{gardner2002brunn} says that
$$
\Abs{\supp f} = \Abs{D_1+\cdots+D_N} \ge \left( \Abs{D_1}^{1/d}+\cdots+\Abs{D_N}^{1/d}\right)^d \ge N^d.
$$
This lower bound
$$
\Abs{\supp f} \ge C N^d
$$
clearly holds also for functions of the form
\beql{convolution-soft}
f = f_1 * f_2 * \cdots * f_N,\ \ \ f_i \ge 0,
\eeq
where for all $i=1,2,\ldots,N$ we assume that the \textit{nonnegative} function $f$ tiles with $\Lambda_i$.

We have proved \cite{kolountzakis2021functions}:
\begin{theorem}\label{th:volume-lb}
For any collection of lattices $\Lambda_1, \ldots, \Lambda_N$ in $\RR^d$ of volume at least 1 and any common tile $f$ for them of the form
$$
f = f_1 * f_2 * \cdots * f_N,\ \ \ f_i \ge 0,
$$
with $f_i$ tiling with $\Lambda_i$, we have
$$
\Abs{\supp f} \ge N^d.
$$
\end{theorem}

But when the functions $f$ are signed (or complex) we only have
$$
\supp{f} \subseteq \supp{f_1}+\cdots\supp{f_N},
$$
not necessarily equality, which brings us to the next question.

\qq{
If $f$ is given by \eqref{convolution-soft}, is it true that
\beql{volume-lb}
\Abs{\supp f} \ge C N^d?
\eeq
}

If one requires that the lattices $\Lambda_1, \Lambda_2, ..., \Lambda_N\subset \mathbb{R}^d$ have the same volume, say $1$, and the sum $\Lambda_1^{\ast}+ \Lambda_2^{\ast}+ ...+ \Lambda_N^{\ast}$ of their dual lattices is direct, then, by \cite[Theorem 2]{kolountzakis1997multi}, they possess a measurable common almost fundamental domain $E$ (generally unbounded).See Fig.\ \ref{fig:multi-construction}.
\begin{figure}[h]
\resizebox{9cm}{!}{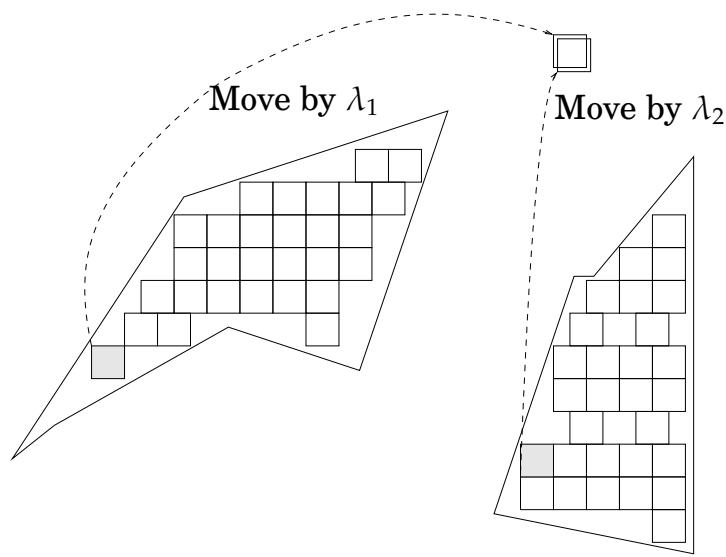}
\caption{How to rearrange the fundamental domains of two lattices so that they agree almost everywhere \cite[Theorem 2]{kolountzakis1997multi}. One breaks up the two domains into smaller and smaller parts, then moves each by a vector in its own lattice so that they agree almost completely.}
\label{fig:multi-construction}
\end{figure}
In this case, $|E|=\vol(\Lambda_i)=1$. So then one can take $f=\one_E$, which tiles with all $\Lambda_i$, $i=1,2, ..., N$, with $|\supp f|=|E|=1$.

Motivated by the previous observation, but now dropping the equal volume assumption, we ask the following:

\qq{
Consider the lattices $\Lambda_1, \Lambda_2, ..., \Lambda_N$, with $\frac{1}{2}\leq \vol(\Lambda_i)\leq 2$. Is there a function $f$ that tiles with all $\Lambda_i$, such that
	\[|\supp f|=o(N^d)?
	\]
}

\qq{
In the case when $\Lambda_1, \ldots, \Lambda_N$ all have volume 1 and satisfy some sort of genericity condition, such as $\Lambda_1^*+\cdots\Lambda_N^*$ being a direct sum, as in \cite[Theorem 2]{kolountzakis1997multi}, can the common fundamental domain of the $\Lambda_i$ be bounded? In the construction of \cite[Theorem 2]{kolountzakis1997multi} the unboundedness is unavoidable, but is it in the nature of things?
}

\mysection{Small length of the support in $d=1$.}
In the simplest case in dimension $d=1$, and for two lattices only, a basic question is to ask if the convolution \eqref{convolution-soft} is best in terms of the length of the support. Here we can give \cite{kolountzakis2021functions} a simple lower bound assuming a nonnegative function.

\begin{theorem}\label{th:1d-lb}
Suppose the nonnegative $f:\RR\to\RR^{\ge 0}$ is measurable and tiles with both $\Lambda_1 = \ZZ$ and with $\Lambda_2 = \alpha \ZZ$, where $\alpha \in (0, 1)$:
\beql{tiling-with-two}
\sum_{n \in \ZZ} f(x-n) = 1,\ \ \ \sum_{n \in \ZZ} f(x-n\alpha) = \frac{1}{\alpha},\ \ \text{for almost every $x \in \RR$}.
\eeq
Then
\beql{1d-lb}
\Abs{\supp f} \ge \Ceil{\frac{1}{\alpha}} \alpha \ge 2\alpha.
\eeq
\end{theorem}

\begin{remark}
If we assume the first equation in \eqref{tiling-with-two} then the constant in the second equation is forced to be $1/\alpha$. This is because $\int f = 1$ (from the first equation), so repeating $f$ at a set of translates of density $1/\alpha$ will give a constant (assuming it tiles) at that level.
\end{remark}

\begin{remark}
Notice that if $\alpha$ is just a little less than 1 then \eqref{1d-lb} gives a lower bound of $2\alpha$, which shows that the convolution $\one_{[0, 1]}*\one_{[0, \alpha]}$ is almost optimal in this case, having support of size $1+\alpha$.

But if, on the other hand, $\alpha$ is just over $1/2$ then the lower bound is just over 1 but the convolution upper bound is just over $3/2$, a considerable gap.
\end{remark}

\begin{proof}
From the first equation in \eqref{tiling-with-two} it follows that $f(x) \le 1$ for almost every $x$. For the second equation to be true it therefore follows that for almost every $x \in \RR$ there are at least $\Ceil{1/\alpha}$ different values of $n \in \ZZ$ such that $f(x-n\alpha)>0$.
Using this for almost all $x \in [0, \alpha)$ (which ensures that for different $x$ the locations $x-n\alpha$ are also different) gives \eqref{1d-lb}.

\end{proof}

\qq{
What is the least possible length of the support of $f$ for a nonnegative $f$ that tiles with both $\ZZ$ and $\alpha\ZZ$?
%See Remark \ref{rem:convolution-length} for a special case.
}

\mysection{Very small diameter of the support with relations among the lattices.}
If we have $N$ lattices
$$
\Lambda_1, \ldots, \Lambda_N \subseteq \RR^d
$$
we can find a function that tiles with them all, namely the function $f$ in \eqref{convolution}. If our lattices are assumed to each have a fundamental domain bounded by $\sim 1$ then $\diam\supp f = O(N)$, and this cannot be improved for functions $f$ arising from \eqref{convolution}. We show here \cite{kolountzakis2021functions} that we can choose the lattices $\Lambda_j$ so that a common tiling function exists which is much more tight than that, tighter even than what Theorem \ref{kw-lb} imposes. Of course our lattices will not satisfy the genericity condition \eqref{no-intersection} of Theorem \ref{kw-lb}, but will satisfy a lot of relations (their intersection will be a large lattice, in terms of density).

Fix a large prime $p$ and consider the group $\ZZ_p^d$. Any nonzero element $g$ of this group generates a cyclic subgroup of order $p$. It follows that $\ZZ_p^d$ has
$$
\frac{p^d-1}{p-1} \sim p^{d-1} =: N
$$
different cyclic subgroups. For each such subgroup $G$, which we now view as a subset of $\Set{0, 1, \ldots, p-1}^d$, consider the lattice
$$
\Lambda_G = (p\ZZ)^d + G.
$$
\begin{figure}[h]
\includegraphics[width=0.3\linewidth]{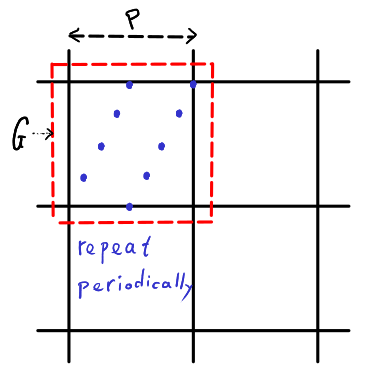}
\caption{How we construct the $(p\ZZ)^d$-periodic set $\Lambda_G$ from the subgroup $G$ of $\ZZ_p^d$}
\end{figure}
This contains the lattice $\Lambda = (p\ZZ)^d$ and has volume
$$
\vol \Lambda_G = \frac{\vol (p\ZZ)^d}{\Abs{G}} = p^{d-1} = N.
$$
The function $f = \one_{[0, p)^d}$, $[0,p)^d$ being a fundamental domain of $\Lambda$, tiles with $\Lambda$ and, therefore, with any larger group, so $f$ is a common tile of all $\Lambda_G$.

In order to make the volume of the $\Lambda_G$ equal to 1 we shrink everything by $N^{1/d}$:
$$
\Lambda_G' = N^{-1/d} \Lambda_G,\ \ \ f'(x) = f(N^{1/d} x).
$$
So we have $\sim N$ lattices $\Lambda'_G$ of volume 1 and a common tile $f'$ for them with
$$
\diam \supp f' = \diam\supp f \cdot N^{-1/d} = \sqrt{d} \, p N^{-1/d} = \sqrt{d}\, N^{\frac{1}{d-1}-\frac{1}{d}} = \sqrt{d} \, N^{\frac{1}{d(d-1)}}.
$$

We have proved:
\begin{theorem}\label{th:many-relations}
In dimension $d\ge 2$ and for arbitrarily large $N$ we can find $N$ lattices of volume $1$ and a common tile $f$ for them with
$$
\diam \supp f = O_d\left( N^{\frac{1}{d(d-1)}} \right),
$$
and, consequently, with
\beql{many-relations-estimate}
\Abs{\supp f} = O_d\left(N^{\frac{1}{d-1}}\right).
\eeq
\end{theorem}

\qq{
Derive a lower bound for $\diam \supp f$, for $f$ tiling with $\Lambda_1, \ldots, \Lambda_N \subseteq \RR^d$ and with $f \ge 0$ (or just $\int f > 0$) under no algebraic conditions for the lattices $\Lambda_j$, assuming only that $\vol\Lambda_j \sim 1$.
}

\qq{
In Theorem \ref{th:many-relations} we have used the cyclic subgroups of $\ZZ_p^d$ because they are easier to count. However the same argument could be carried out using a larger class of subgroups, perhaps all of them. What is the estimate that can be achieved this way to replace \eqref{many-relations-estimate}?
}

\mysection{Almost matching upper and lower bounds for the diameter, $d=1$.}
The construction that we used to prove Theorem \ref{th:many-relations} gives nothing in dimension $d=1$. Yet, we can prove \cite{kolountzakis2021functions} that, if we allow relations among the lattices, we can achieve $\diam\supp f = o(N)$ in dimension 1 as well.

Let us start by defining
$$
\lambda_j = \frac{1}{N+j},\ \ \ \Lambda_j = \lambda_j\ZZ, \ \ \ (j=1, 2, \ldots, N).
$$
We will first construct a function $f$ which tiles with all the $\Lambda_j$, $j=1, 2, \ldots, N$, such that
$$
\diam\supp f = o(1).
$$
The Fourier transform of such an $f$ must vanish on the dual lattices
$$
\Lambda_j^* = \lambda_j^{-1}\ZZ = (N+j)\ZZ,\ \ \ (j = 1, 2, \ldots, N)
$$
except at 0.
Write
$$
U = \bigcup_{j=1}^N (N+j)\ZZ \ \setminus\Set{0}.
$$
By a result of Erd\H os \cite{erdos1935note} $U$, the set of integers which are divisible by one of the integers in $\Set{N+1, N+2, \ldots, 2N}$, has density tending to 0 with $N$. Tenenbaum \cite{tenenbaum1980lois} has given the estimate that this density is at most
\beql{ten}
\frac{1}{\log^{\delta - o(1)} N},
\eeq
where $\delta = 0.086071\cdots$ is an explicit constant.

It is an important result of Beurling \cite{beurling1989collected} that if $\Lambda$ is a uniformly discrete set of real numbers of upper density $\rho$ then for any $\epsilon>0$ we can find a continuous function $f$, not identically zero, supported by the interval $[0, \rho+\epsilon]$ such that $\ft{f}(\lambda) = 0$ for all $\lambda \in \Lambda$. We can even ask that $\ft{f}(0) = 1$ if $0 \notin \Lambda$. By Tenenbaum's estimate \eqref{ten} we can take $\rho = \log^{-\delta+o(1)}N$ and the set $U$, being a set of integers and thus uniformly discrete, satisfies the assumptions of Beurling's theorem, so there is a function $f$ supported in the interval $[0, \log^{-\delta+o(1)}N]$, with integral 1, such that $\ft{f} = 0$ on $U$. It follows that $f$ tiles with all $\Lambda_j$.

We now scale by a factor of $N$
$$
f'(x) = f(x / N),\ \ \ \Lambda_j' = N \Lambda_j,\ \ \ \diam\supp f' = O(N \log^{-\delta+o(1)}N)
$$
and obtain the first half of the following theorem.
\begin{theorem}\label{th:many-relations-1d}
We can find $N$ lattices $\Lambda_j \subseteq \RR$ of with $\vol\Lambda_j \sim 1$ and a function $f$ with $\int f > 0$ and supported in an interval of length
$$
\frac{N}{\log^{\delta-o(1)}N)}
$$
which tiles with all $\Lambda_j$.

Furthermore, for any $\epsilon > 0$ any such function $f$ must have
$$
\diam\supp f \gtrsim_\epsilon N^{1-\epsilon}.
$$
\end{theorem}

Arguing similarly we can also prove the lower bound for $\diam\supp f$ in Theorem \ref{th:many-relations-1d}. If we assume that $f$ tiles with all $\Lambda_j = \lambda_j \ZZ$, with, say, $1 \le \lambda_j \le 2$, $j=1, 2, \ldots, N$, then $\ft{f}$ vanishes on
$$
\bigcup_{j=1}^N \lambda_j^{-1}\ZZ\ \ \setminus \Set{0}.
$$
If this set is large then Jensen's formula implies that $\diam\supp f$ is also large.
It was proved in \cite[Theorem 1.1, special case $\ell= n$]{gilboa2014union} that, for any $\epsilon > 0$, the above union of arithmetic progressions contains at least $c_\epsilon N^{2-\epsilon}$ points in $[0, 2N]$. By Jensen's formula then we have $\diam\supp f \gtrsim_\epsilon N^{1-\epsilon}$ and this completes the proof of Theorem \ref{th:many-relations-1d}.

\qq{
Can we ensure $f \ge 0$ in the first half of Theorem \ref{th:many-relations-1d}?
}

\bibliographystyle{amsplain}
\bibliography{lit}

\providecommand{\bysame}{\leavevmode\hbox to3em{\hrulefill}\thinspace}
\providecommand{\MR}{\relax\ifhmode\unskip\space\fi MR }
\providecommand{\MRhref}[2]{%
  \href{http://www.ams.org/mathscinet-getitem?mr=#1}{#2}
}
\providecommand{\href}[2]{#2}
\begin{thebibliography}{10}

\bibitem{beck1989lattice}
J.~Beck, \emph{{On a lattice point problem of H. Steinhaus}}, Studia Sci. Math.
  Hung \textbf{24} (1989), 263--268.

\bibitem{beurling1989collected}
A~Beurling, \emph{{Collected Works of Arne Beurling (2 vol.), edited by L.
  Carleson et al}}, 1989.

\bibitem{chan2007steinhaus}
W.~K. Chan and R.~Mauldin, \emph{{Steinhaus tiling problem and integral
  quadratic forms}}, Proceedings of the American Mathematical Society
  \textbf{135} (2007), no.~2, 337--342.

\bibitem{ciucu1996remark}
M.~Ciucu, \emph{{A remark on sets having the Steinhaus property}},
  Combinatorica \textbf{16} (1996), no.~3, 321--324.

\bibitem{croft1982three}
H.~T. Croft, \emph{{Three lattice-point problems of Steinhaus}}, The Quarterly
  Journal of Mathematics \textbf{33} (1982), no.~1, 71--83.

\bibitem{erdos1935note}
Paul Erdos, \emph{Note on sequences of integers no one of which is divisible by
  any other}, J. London Math. Soc \textbf{10} (1935), no.~1, 126--128.

\bibitem{gardner2002brunn}
R.~Gardner, \emph{{The Brunn-Minkowski inequality}}, Bulletin of the American
  Mathematical Society \textbf{39} (2002), no.~3, 355--405.

\bibitem{gilboa2014union}
S.~Gilboa and R.~Pinchasi, \emph{On the union of arithmetic progressions}, SIAM
  Journal on Discrete Mathematics \textbf{28} (2014), no.~3, 1062--1073.

\bibitem{havin1994uncertainty}
V.~Havin and B.~J{\"o}ricke, \emph{{The Uncertainty Principle in Harmonic
  Analysis}}, Springer, Berlin, 1994.

\bibitem{jackson2002lattice}
S.~Jackson and R.~Mauldin, \emph{{On a lattice problem of H. Steinhaus}},
  Journal of the American Mathematical Society \textbf{15} (2002), no.~4,
  817--856.

\bibitem{jackson2002sets}
S.~Jackson and R.~D. Mauldin, \emph{{Sets meeting isometric copies of the
  lattice $\ZZ^2$ in exactly one point}}, Proceedings of the National Academy
  of Sciences \textbf{99} (2002), no.~25, 15883--15887.

\bibitem{jackson2003survey}
\bysame, \emph{{Survey of the Steinhaus tiling problem}}, Bulletin of Symbolic
  Logic \textbf{9} (2003), no.~03, 335--361.

\bibitem{kolountzakis1996new}
M.~N. Kolountzakis, \emph{{A new estimate for a problem of Steinhaus}},
  International Mathematics Research Notices \textbf{1996} (1996), no.~11,
  547--555.

\bibitem{kolountzakis1996problem}
\bysame, \emph{{A problem of Steinhaus: Can all placements of a planar set
  contain exactly one lattice point?}}, Progress in Mathematics \textbf{139}
  (1996), 559--566.

\bibitem{kolountzakis1997multi}
\bysame, \emph{Multi-lattice tiles}, International Mathematics Research Notices
  \textbf{1997} (1997), no.~19, 937--952.

\bibitem{kolountzakis2004milano}
\bysame, \emph{{The study of translational tiling with Fourier Analysis}},
  Fourier Analysis and Convexity (L.~Brandolini, ed.), Birkh\"auser, 2004,
  pp.~131--187.

\bibitem{kolountzakis2002steinhaus}
M.~N. Kolountzakis and M.~Papadimitrakis, \emph{{The Steinhaus tiling problem
  and the range of certain quadratic forms}}, Illinois Journal of Mathematics
  \textbf{46} (2002), no.~3, 947--951.

\bibitem{kolountzakis2017measurable}
\bysame, \emph{Measurable steinhaus sets do not exist for finite sets or the
  integers in the plane}, Bulletin of the London Mathematical Society
  \textbf{49} (2017), no.~5, 798--805.

\bibitem{kolountzakis1999steinhaus}
M.~N. Kolountzakis and T.~Wolff, \emph{{On the Steinhaus tiling problem}},
  Mathematika \textbf{46} (1999), no.~02, 253--280.

\bibitem{kolountzakis2021tiling}
Mihail~N Kolountzakis and Nir Lev, \emph{Tiling by translates of a function:
  results and open problems}, Discrete Analysis (2021), 28122.

\bibitem{kolountzakis2021functions}
Mihail~N Kolountzakis and Effie Papageorgiou, \emph{Functions tiling with
  several lattices}, arXiv preprint arXiv:2106.11701 (2021).

\bibitem{komjath1992lattice}
P.~Komj{\'a}th, \emph{{A lattice-point problem of Steinhaus}}, The Quarterly
  Journal of Mathematics \textbf{43} (1992), no.~2, 235--241.

\bibitem{mallinikova1995}
E.~Mallinikova and S.~Rukshin, \emph{{On one Steinhaus problem}}, Vestnik
  St.Petersburg Univ. Math. \textbf{28} (1995), no.~1, 28--32.

\bibitem{mauldin2003comments}
R.~Mauldin and A.~Yingst, \emph{{Comments about the Steinhaus tiling problem}},
  Proceedings of the American Mathematical Society \textbf{131} (2003), no.~7,
  2071--2079.

\bibitem{mauldin2001some}
R.~D. Mauldin, \emph{{Some problems in set theory, analysis and geometry}},
  Paul Erdos and his Mathematics I (2001), 493--505.

\bibitem{sierpinski1958probleme}
W.~Sierpi{\'n}ski, \emph{{Sur un probleme de H. Steinhaus concernant les
  ensembles de points sur le plan}}, Fundamenta Mathematicae \textbf{2} (1958),
  no.~46, 191--194.

\bibitem{srivastava2005steinhaus}
S.~M. Srivastava and R.~Thangadurai, \emph{{On Steinhaus sets}}, Expositiones
  Mathematicae \textbf{23} (2005), no.~2, 171--177.

\bibitem{tenenbaum1980lois}
G.~Tenenbaum, \emph{Lois de r{\'e}partition des diviseurs}, S{\'e}minaire
  Delange-Pisot-Poitou. Th{\'e}orie des nombres \textbf{19} (1980), no.~1,
  1--3.

\end{thebibliography}

\end{document}